\documentclass{article}
\usepackage{spconf,amsmath,graphicx}
\usepackage{amsfonts}
\usepackage{algorithm}
\usepackage{algorithmic}

\usepackage{epsfig}
\usepackage{subfig}
\newtheorem{theorem}{Theorem}

\newtheorem{proof}{Proof}

\title{Iteratively reweighted penalty alternating minimization methods  with continuation for  image deblurring}
\name{Tao Sun$^{\dagger}$ \qquad Dongsheng Li$^{\dagger}$ \qquad Hao Jiang$^{\dagger}$ \qquad Zhe Quan$^{\sharp}$\thanks{This work is supported in part by  National Basic Research Program of China (2014CB340303), National Science Foundation of China (61402495 and 61571008), and National Natural Science Foundation of Hunan Province in China (2018JJ3616), and National Natural Science Foundation for the Youth of China (61602166). Zhe Quan is the corresponding author. }}
\address{$^{\dagger}$ College of Computer, National University of Defense Technology,\\ \textit{nudtsuntao@163.com},\{dsli,haojiang\}@nudt.edu.cn\\
$^{\sharp}$ College of Computer Science and Electronic Engineering, Hunan University,\textit{ quanzhe@hnu.edu.cn}}

\begin{document}

\maketitle
\begin{abstract}
In this paper, we consider a class of nonconvex problems with linear constraints appearing frequently in the area of image processing. We solve this problem by the penalty  method and propose the iteratively reweighted alternating minimization algorithm. To speed up the algorithm, we also apply  the continuation strategy to the penalty parameter. A convergence result is proved for the algorithm. Compared with the nonconvex ADMM, the proposed algorithm enjoys both theoretical and computational advantages like weaker convergence requirements and faster speed.    Numerical results demonstrate the efficiency  of the proposed algorithm.
\end{abstract}
\section{Introduction}
Linearly constrained problems are widely  discussed through various disciplines  such as image sciences, signal processing and machines learning, to name a few. The classical algorithm for the linearly constrained problems is the Alternating Direction Method of Multiplier (ADMM); and the previous literature has paid their attention to the convex case \cite{eckstein2015understanding,fortin2000augmented,gabay1976dual,glowinski1985numerical}. In recent years, nonconvex ADMM has been developed for the nonconvex problems \cite{xu2013block,wang2015global,hong2016convergence,wang2015convergence}.
 \subsection{Motivations}
 Although ADMM can be applied to the nonconvex Total Variation (TV) deblurring problem, several drawbacks still exist.  We point out three   of them as follows.
\begin{enumerate}
\item   The convergence guarantees of nonconvex ADMMs require  a very large Lagrange dual multiplier.  Worse still, the  large multiplier   makes the nonconvex ADMM  run slowly.
\item    When applying nonconvex ADMMs to the nonconvex TV deblurring model, by direct checks, the convergence requires TV operator to be full row-rank; however, the TV operator cannot promise such an assumption. This point has been proposed in \cite{sun2017iteratively}.
\item   The previous analyses show that the sequence converges to a critical point of an auxiliary function under several assumptions. But  the relationship between the  auxiliary function and the original one is unclear in the nonconvex settings.
\end{enumerate}
Considering these drawbacks, from both computational and theoretical perspectives, it is necessary to consider novel and efficient solvers.   The main reason why nonconvex ADMMs have these drawbacks is due to the dual variable; in the convergence proof of the nonconvex case, the  dual variables are just simply processed by   Cauchy inequalities; the deductions in the proofs are then somehow loose. Therefore,   we consider employing the penalty method to avoid using the dual information.

\subsection{Contributions and organization}
In this paper, we consider using the penalty method for a class of nonconvex linearized constrained minimizations. Different from the nonconvex ADMM, determining the penalty  multiplier in the proposed algorithm is very lowly-costly. Although  the penalty  multiplier is also large, we can use a continuation method, i.e., increasing the penalty  multiplier in the iteration. The alternating minimization methods \cite{beck2015convergence,sun2017little} are fit for solving this penalty problem. Directly applying the alternating minimization for the penalty problem encounters an  issue: the subproblem may have no closed form.  To overcome this problem, combining the structure of the problem, we use the linearized techniques for the regularized part in the algorithm. In this way, all the subproblems are convex and can be minimized numerically globally even without enjoying a closed form solution. We proved the square-summability    of the successive differences of the generated points. We apply our algorithm to the nonconvex image delurring problem and compare it with the nonconvex ADMM. The numerics show the efficiency and speed  of the proposed algorithm.

In Section 2, we present our problem and algorithm, and the convergence results of the algorithm. Section 3 contains the applications and numerics. And then Section 4 concludes the paper.
\section{Problem formulation and algorithm}
In this paper, we consider a broad class of nonconvex and nonsmooth problems with the following form:
\begin{equation}\label{model}
    \min_{x,y}\{\Psi(x,y):=f(x)+\sum_{i=1}^{N}h(g(y_i)),~\textrm{s.t.}~Ax+By=c\}.
\end{equation}
where $x\in \mathbb{R}^M,y\in \mathbb{R}^N$, and functions $f$, $g$ and $h$ satisfy the following assumptions:
\begin{itemize}
\item \textbf{A.1} $f:\mathbb{R}^M\rightarrow \mathbb{R}$ is a closed proper convex function and $\inf_{x\in\mathbb{R}^M}f(x)>-\infty$.
\item \textbf{A.2} $g:\mathbb{R}\rightarrow \mathbb{R}$ is a convex function, and the proximal map of $g$ is easy to calculated.\footnote{We say the proximal map of $g$ is easy to calculate  if the minimization problem $\textbf{Prox}_{g}(d):=\textrm{arg}\min_{t\in\mathbb{R}}\{g(t)+\frac{1}{2}|t-d|^2\}$ can be solved very easily for any $d\in\mathbb{R}$.}
\item \textbf{A.3} $h:\textrm{Im}(g)\rightarrow \mathbb{R}$ is a concave function and $\inf_{t\in\textrm{Im}(g)}h(t)>-\infty$.
\end{itemize}

A very classical problem which can be formulated as \eqref{model} is Total Variation $q$ (TV-$q$) deblurring \cite{hintermuler2013nonconvex}
\begin{equation}\label{tvq}
    \min_{u}\{\frac{1}{2}\|H(u)-B\|_F^2+\lambda\|T(u)\|_q^q\},
\end{equation}
where $H$ is the blurring operator, $T$ is the well-known total variation  operator and $q\in(0,1)$. By defining $v:=T(u)$, the problem then turns to
\begin{equation}\label{ctvq}
    \min_{u,v}\{\frac{1}{2}\|H(u)-B\|_F^2+\lambda\|v\|_q^q~~\textrm{s.t.}~~v=T(u)\}.
\end{equation}

\subsection{Algorithm}
We consider the penalty function as
\begin{equation}\label{almodel}
   \min_{x,y}\{\Phi_{\gamma}(x,y):=f(x)+\sum_{i=1}^{N}h(g(y_i))+\frac{\gamma}{2}\|Ax+By-c\|_2^2\}.
\end{equation}
The  difference between problem (\ref{model}) and (\ref{almodel})  is determined by the parameter $\gamma$. They are identical if $\gamma=+\infty$.
Assume that $(x^*,y^*)$ is the solution to problem (\ref{almodel}), and $(x^{\dag},y^{\dag})$ is the solution to problem (\ref{model}),  and $(\hat{x},\hat{y})$ is any one satisfying $A\hat{x}+B\hat{y}=c$, we then have the following claims.
\begin{equation}\label{claim1}
    \Psi(x^{*},y^{*})\leq \Psi(x^{\dag},y^{\dag})
\end{equation}
and
\begin{equation}\label{claim2}
  \|Ax^{*}+By^{*}-c\|_2^2\leq\frac{2}{\gamma}[ \Psi(\hat{x},\hat{y})-\underline{f}-\underline{h_{g}}],
\end{equation}
where $\inf_{x\in\mathbb{R}^M}f(x)\geq \underline{f}$ and $\inf_{t\in\textrm{Im}(g)}h(t)\geq \underline{h_{g}}$. These two claims can provide the errors between (\ref{model}) and (\ref{almodel}).
We present brief proofs for claims \eqref{claim1} and \eqref{claim2}.
First,  with the definition of $(x^{*},y^{*})$, we have
\begin{equation}\label{diff-t1}
    \Phi_{\gamma}(x^{*},y^{*})\leq\Phi_{\gamma}(x^{\dag},y^{\dag}).
\end{equation}
Noting $Ax^{\dag}+By^{\dag}=c$,
 we then have
 \begin{equation}
    \Psi(x^{*},y^{*})+\frac{\gamma}{2}\|Ax^{*}+By^{*}-c\|_2^2\leq \Psi(x^{\dag},y^{\dag}).
 \end{equation}
 Thus, we are led to
 \begin{equation}
   \Psi(x^{*},y^{*})\leq \Psi(x^{\dag},y^{\dag}).
 \end{equation}
 Similarly, we derive
  \begin{equation}
    \Psi(x^{*},y^{*})+\frac{\gamma}{2}\|Ax^{*}+By^{*}-c\|_2^2\leq \Psi(\hat{x},\hat{y}).
 \end{equation}
  With the fact $\Psi(x^{*},y^{*})\geq \inf_{x\in\mathbb{R}^M}f(x)+\inf_{t\in\textrm{Im}(g)}h(t)$,
we then get \eqref{claim2}.

We apply the claims to the    TV deblurring problem \eqref{ctvq}, we can see $\inf_{u}\|H(u)-B\|_2^2\geq 0, \inf_{v}\|v\|_q^q\geq 0$; and we can choose $\hat{u}=\textbf{0}$ and $\hat{v}=\textbf{0}$. Then, it holds
$$\|Tu^{*}-v^{*}\|_2^2\leq\frac{2\|B\|^2_2}{\gamma},$$
where $(u^{*},v^{*})$ is the minimizer of the penalty problem.
Thus, to achieve $\varepsilon$ error approximation, we just need to set $\gamma=\frac{2\|B\|^2_2}{\varepsilon}$.

The classical algorithm solving this problem is the Alternating Minimization (AM) method, i.e., minimizing one variable while fixing the other one. However, if directly applying AM to model (\ref{almodel}), the subproblem may still be nonconvex; the minimizer is hard to obtain in most cases. Considering the structure of the problem, we use a linearized technique for the  nonsmooth part $\sum_{i=1}^{N}h(g(y_i))$. This method was inspired by the reweighted algorithms \cite{chartrand2008iteratively,candes2008enhancing,chen2014convergence,sun2017global}. To derive the sufficient descent, we also add a proximal term. We call it as  iteratively reweighted penalty alternating minimization (IRPAM) method  which can be described as
\begin{equation}\label{scheme}
    \left\{\begin{array}{c}
             x^{k+1}\in \textrm{arg}\min_{x} \{f(x)+\frac{\gamma}{2}\|Ax+By^k-c\|_2^2\},\\
             y^{k+1}\in \textrm{arg}\min_{y} \{\sum_{i}^N w^{k}_i g(y_i)\\
             +\frac{\gamma}{2}\|Ax^{k+1}+By-c\|_2^2+\frac{\delta\gamma\|y-y^k\|_2^2}{2}\},
           \end{array}
    \right.
\end{equation}
where $w^{k}_i\in -\partial(-h(g(y_i^k)))$, $i=1,2,\ldots,N$. If $h(t)=t$, $-\partial(-h(g(\cdot)))=1$; the algorithm is actually the AM. In the algorithm, all subproblems are convex. If the proximal maps of $g$ and $f$ are easy to calculate, both  subproblems are easy to solve. If $B=I$, the minimizer of the second problem reduces to the following form
\begin{equation}
    y^{k+1}_i=\textbf{Prox}_{w_i^{k+1}g/(1+\delta)\gamma}(\frac{\delta y^k_i}{1+\delta}+\frac{c_i}{1+\delta}-\frac{A_i x^{k+1}}{1+\delta}),
\end{equation}
where $A_i$ is the $i$-th row of $A$, and  $i=1,2,\ldots,N$. Actually, in the TV deblurring model, $B$ is identical map.
When implementing the algorithms, we increase $\gamma$ in each iteration and set an upper bound $\bar{\gamma}$. This  continuation technique was used in   \cite{hale2008fixed,sun2016convergence,sunhard}. In the continuation version, we use  $\gamma_k$ rather than   constant $\gamma$ in the $k$-th iteration.  And the scheme of IRPAM(C) can be presented as follows.
\begin{algorithm}
\caption{Iteratively reweighted penalty alternating minimization  (with continuation)}
\begin{algorithmic}
\REQUIRE   parameters $\bar{\gamma}>0,a>1,\delta>0$ \\
\textbf{Initialization}: $z^{0}=(x^0,y^0)$, $\gamma_0>0$\\
\textbf{for} $k=0,1,2,\ldots$ \\
~~~ $x^{k+1}\in \textrm{arg}\min_{x} \{f(x)+\frac{\gamma_k}{2}\|Ax+By^k-c\|_2^2\}$ \\
~~~ $w^{k}_i\in -\partial(-h(g(y_i^k)))$, $i\in[1,2,\ldots,N]$ \\
~~~ $y^{k+1}\in \textrm{arg}\min_{y} \{\sum_{i}^N w^{k}_i g(y_i)+\frac{\gamma_k}{2}\|Ax^{k+1}+By-c\|_2^2+\frac{\delta\gamma_k\|y-y^k\|_2^2}{2}\}$ \\
~~~ $\gamma_{k+1}= \min\{\bar{\gamma},(a\gamma_k)\}$\\
\textbf{end for}\\
\textbf{Output} $x^k$\\
\end{algorithmic}
\end{algorithm}
In Algorithm 1, if we set $a=1$, the algorithm is indeed the IRPAM; and if  $a>1$, the algorithm is then the IRPAMC.  When   IRPAMC being applied to the TV-q deblurring problem, the subproblems just involve  with FFT and soft-shrinkages which can be solved fast. More details can be founded in \cite{wang2008new}.

Variants of IRPAMC can be developed by using linearization for the quadratic term, or adding the term $\frac{1}{2}\|x-x^k\|^2_2$  in the minimization in the $k$-th iteration, or even  by hybrid way. In  \cite{sun2017little}, the authors introduced various AM schemes which  can be modified for IRPAMC to propose variants.

We have shown that for any given $\varepsilon$, $\gamma$ can be set explicitly. And the convergence of IRPAMC is free of the requirement for the full-rank of $T$. Then,  compared with the nonconvex ADMM, IRPAMC can overcome the three drawbacks pointed out in previous section.
\subsection{Convergence}
In this part, we present the convergence of IRPAMC.  Specifically, we prove that the square of the difference of the generated point is summable. For technical reasons, we need an extra assumption.
\begin{itemize}
\item \textbf{A.4} $f(x)+\frac{1}{2}\|Ax\|^2_2$ is strongly convex with $\nu$.
\end{itemize}
Now, we discuss the validity of Assumption \textbf{A.4}. For the deblurring model \eqref{ctvq}, \textbf{A.4}   actually requires $\|H(u)\|^2_F+\|T(u)\|^2_F$  to be strongly convex.  With basic linear algebra, we just need to verify $\textrm{Null}(H)\bigcap \textrm{Null}(T)=\textbf{0}$.  Direct computing gives us $\textrm{Null}(T)=(1)_{M\times N}:=\textbf{1}_{M\times N}$. For the blurring operator $H$, $H(\textbf{1}_{M\times N})\neq\textbf{0}$. That means \textbf{A.4} holds for the  deblurring model.
\begin{theorem}\label{descend}
Assume that $(z^{k})_{k\geq 0}$ is generated by IRPAMC and Assumptions \textbf{A.1}, \textbf{A.2}, \textbf{A.3} and \textbf{A.4}  hold, and $\delta>0$. Then we have the following results.\\

(1)  It holds that
\small
\begin{align}\label{descend1}
    &\Phi_{\bar{\gamma}}(x^k,y^k)-\Phi_{\bar{\gamma}}(x^{k+1},y^{k+1})\nonumber\\
    &\quad\quad\geq \min\{\bar{\gamma},\nu\bar{\gamma}\}\cdot\|x^{k+1}-x^{k}\|_2^2+\frac{\delta\bar{\gamma}\|y^{k+1}-y^k\|_2^2}{2}.
\end{align}
\normalsize
for $k>K$ with   $K=\ulcorner\log_{a}(\frac{\bar{\gamma}}{\gamma_0})\urcorner$.

(2) $\sum_{k}(\|x^{k+1}-x^{k}\|_2^2+\|y^{k+1}-y^{k}\|_2^2)<+\infty$, which implies  that
\begin{equation}
    \lim_{k}\|x^{k+1}-x^k\|_2=0,\,~\lim_{k}\|y^{k+1}-y^k\|_2=0.
\end{equation}
\end{theorem}
\begin{proof}
(1) The convexity of $-h$ and the fact $-w^{k}_i\in \partial(-h(g(y_i^k)))$ yield
\small
\begin{align}
    &[-h(g(y_i^{k+1}))]-[-h(g(y_i^{k}))]\geq \langle-w^{k}_i,g(y_i^{k+1})-g(y_i^{k})\rangle.
\end{align}
\normalsize
That is also
\begin{equation}\label{descend-t1}
   h(g(y_i^{k}))-h(g(y_i^{k+1}))\geq \langle w^{k}_i,g(y_i^{k})-g(y_i^{k+1})\rangle.
\end{equation}
It is easy to see that $K=\ulcorner\log_{a}(\frac{\bar{\gamma}}{\gamma_0})\urcorner$, $\gamma_k\equiv\bar{\gamma}$ if $k>K$. In the update of $y^{k+1}$, we have
\begin{align}\label{descend-t2}
&\sum_{i}^N w^{k}_i g(y_i^k)+\frac{\bar{\gamma}}{2}\|Ax^{k+1}+By^k-c\|_2^2\nonumber\\
&\quad\geq \sum_{i}^N w^{k}_i g(y_i^{k+1})
             +\frac{\bar{\gamma}}{2}\|Ax^{k+1}+By^{k+1}-c\|_2^2\nonumber\\
&\quad             +\frac{\delta\bar{\gamma}\|y^{k+1}-y^k\|_2^2}{2}.
\end{align}
Combining \eqref{descend-t1} and \eqref{descend-t2}, we then derive
\begin{align}\label{descend-t2+}
&\sum_{i}^N  h(g(y_i^{k}))+\frac{\bar{\gamma}}{2}\|Ax^{k+1}+By^k-c\|_2^2\nonumber\\
&\quad\geq \sum_{i}^N h(g(y_i^{k+1}))
             +\frac{\bar{\gamma}}{2}\|Ax^{k+1}+By^{k+1}-c\|_2^2\nonumber\\
             &\quad             +\frac{\delta\bar{\gamma}\|y^{k+1}-y^k\|_2^2}{2}.
\end{align}
That is also
\begin{eqnarray}\label{descend-t3}
\Phi_{\bar{\gamma}}(x^{k+1},y^k)-\Phi_{\bar{\gamma}}(x^{k+1},y^{k+1})\geq \frac{\delta\bar{\gamma}\|y^{k+1}-y^k\|_2^2}{2}.
\end{eqnarray}
With Assumption \textbf{A.4}, $f(x)+\frac{\bar{\gamma}}{2}\|Ax+By^k-c\|_2^2$ is then strongly convex with $\min\{\bar{\gamma},\nu\bar{\gamma}\}$. While   $x^{k+1}$ is the minimizer, the strong convexity the yields
\begin{align}\label{descend-t2'}
 &\frac{\bar{\gamma}}{2}\|Ax^k+By^k-c\|_2^2+f(x^k)\nonumber\\
 &\qquad-\left(\frac{\bar{\gamma}}{2}\|Ax^{k+1}+By^k-c\|_2^2+f(x^{k+1})\right)\nonumber\\
 &\quad\quad\geq \min\{\bar{\gamma},\nu\bar{\gamma}\}\cdot\|x^{k+1}-x^k\|^2_2.
\end{align}
The relation (\ref{descend-t2'}) also means
\small
\begin{align}\label{descend-t4}
&\Phi_{\bar{\gamma}}(x^{k},y^k)-\Phi_{\bar{\gamma}}(x^{k+1},y^{k})\geq \min\{\bar{\gamma},\nu\bar{\gamma}\}\cdot\|x^{k}-x^{k+1}\|_2^2.
\end{align}
\normalsize
Summing (\ref{descend-t3}) and (\ref{descend-t4}), we then get
\eqref{descend1}.

(2) From (\ref{descend1}), $(\Phi_{\bar{\gamma}}(x^{k},y^{k}))_{k\geq K}$ is non-increasing for large $K$. Noting $\inf_k\{\Phi_{\bar{\gamma}}(x^{k},y^{k})\}>-\infty$, we can see $(\Phi_{\bar{\gamma}}(x^{k},y^{k}))_{k\geq 0}$ is convergent. Hence, we can easily have
\begin{align}
&\sum_{j=K}^{k}(\|x^{j+1}-x^{j}\|_2^2+\|y^{j+1}-y^{j}\|_2^2)\nonumber\\
&\qquad\leq\frac{\Phi_{\bar{\gamma}}(x^{K},y^{K})-\Phi_{\bar{\gamma}}(x^{k+1},y^{k+1})}{\min\{\bar{\gamma},\nu\bar{\gamma},\frac{\delta\bar{\gamma}}{2}\}}<+\infty.
\end{align}
\end{proof}
\section{Application to image deblurring}
In this part, we apply the proposed algorithm  to image deblurring and compare the performance with the nonconvex ADMM. The codes of all algorithms are written entirely in MATLAB, and all the experiments are
implemented under Windows  and MATLAB R2016a running on a laptop with an Intel Core
i5 CPU (2.8 GHz) and 8 GB Memory. The Lena
image  is used in the numerical experiments.

 We solve (\ref{ctvq}) when $q=0.5$, and use the nonconvex ADMM proposed in \cite{sun2017iteratively} for comparison.
The performance of the proposed deblurring algorithms is routinely measured by means of the signal-to-noise ratio (SNR)
\begin{equation}
    \textrm{SNR}(u,u^*):=10\lg\left\{\frac{\|u-\bar{u}\|_2^2}{\|u^*-\bar{u}\|_2^2}\right\},
\end{equation}
where $u$ and $u^*$ denote the original image and the deblurring image, respectively, and $\bar{u}$ stands for
the mean of the original image.
 In the experiments, the blurring operators is generated
by the Matlab command   \verb"fspecial('gaussian',.,.)". The blurred image is generated by
\begin{equation}
    B=H(u)+e,
\end{equation}
where $e$ is the Gaussian noise with power of $\sigma$. In the experiment, we set $\sigma=10^{-8}$ and $\lambda=10^6$, and $\delta=10^{-6}$.
The proposed algorithms  are terminated after 200 iterations.
 The parameters are set as  $\gamma_0=10$, $\bar{\gamma}=1000$ and $a=1.1$. We compare IRPAMC with the nonconvex ADMM, in which the  Lagrange dual multiplier is also set as $1000$.  For both algorithms, the initializations are set as the  blurred image. The numerical results are shown in Fig. 1.
\begin{figure}[!htb]
  \centering
    \subfloat[]{
    \includegraphics[width=1.5in]{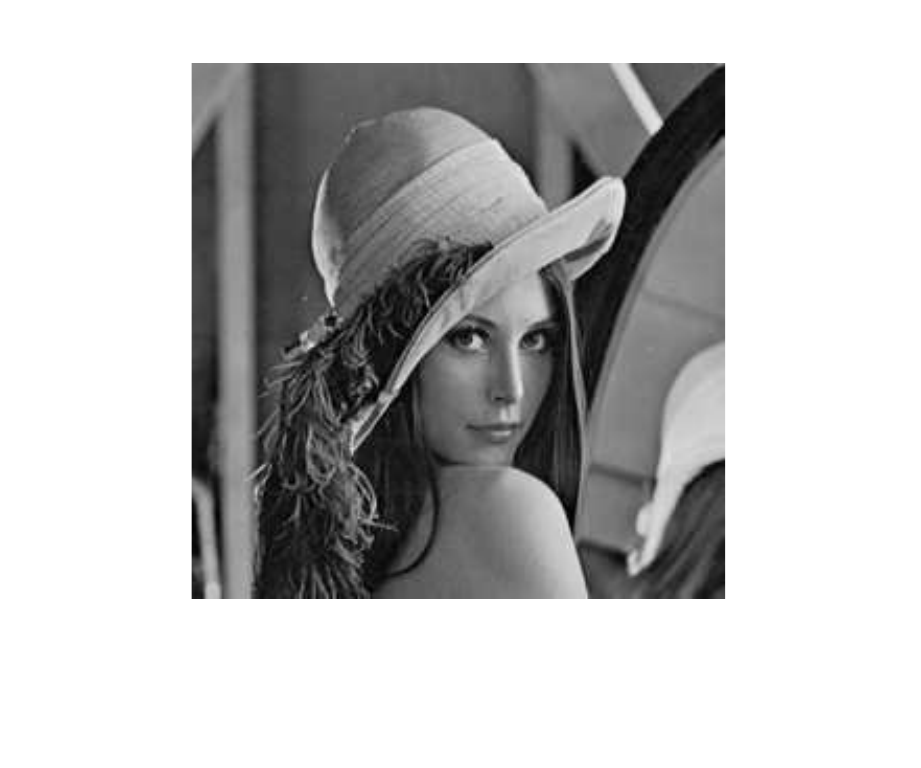}}
  \subfloat[]{
    \includegraphics[width=1.5in]{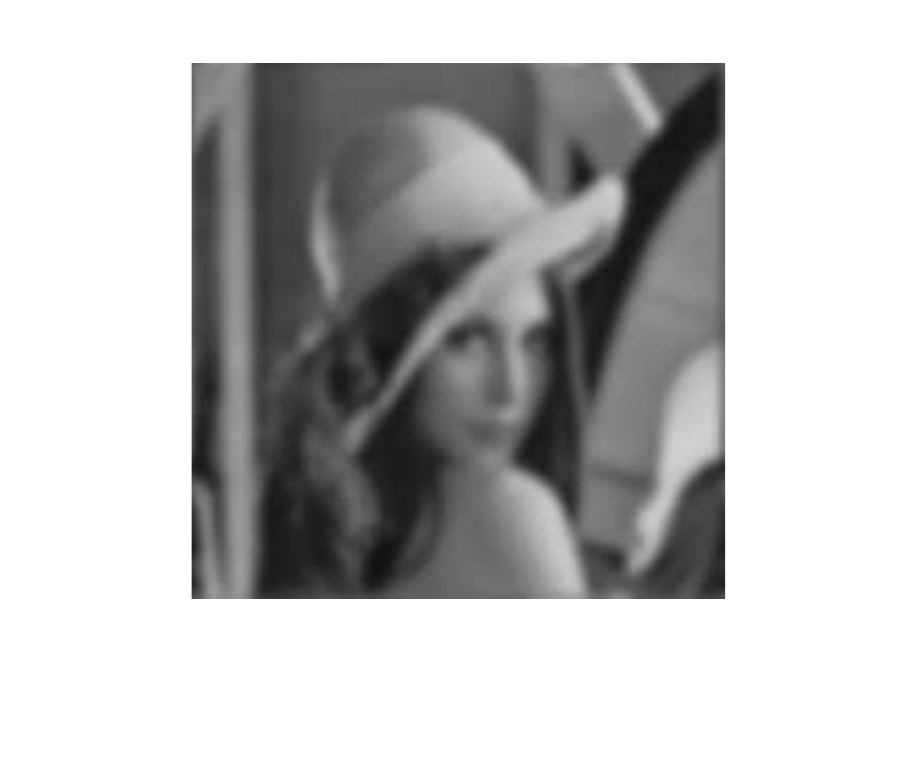}}\\
  \subfloat[]{
    \includegraphics[width=1.5in]{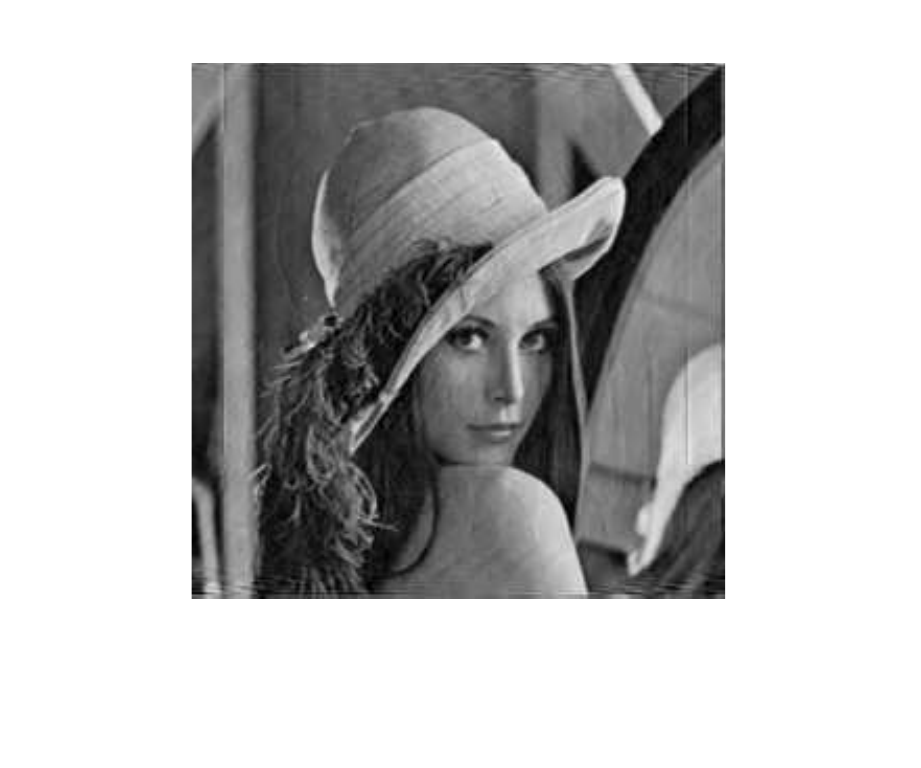}}
      \subfloat[]{
    \includegraphics[width=1.5in]{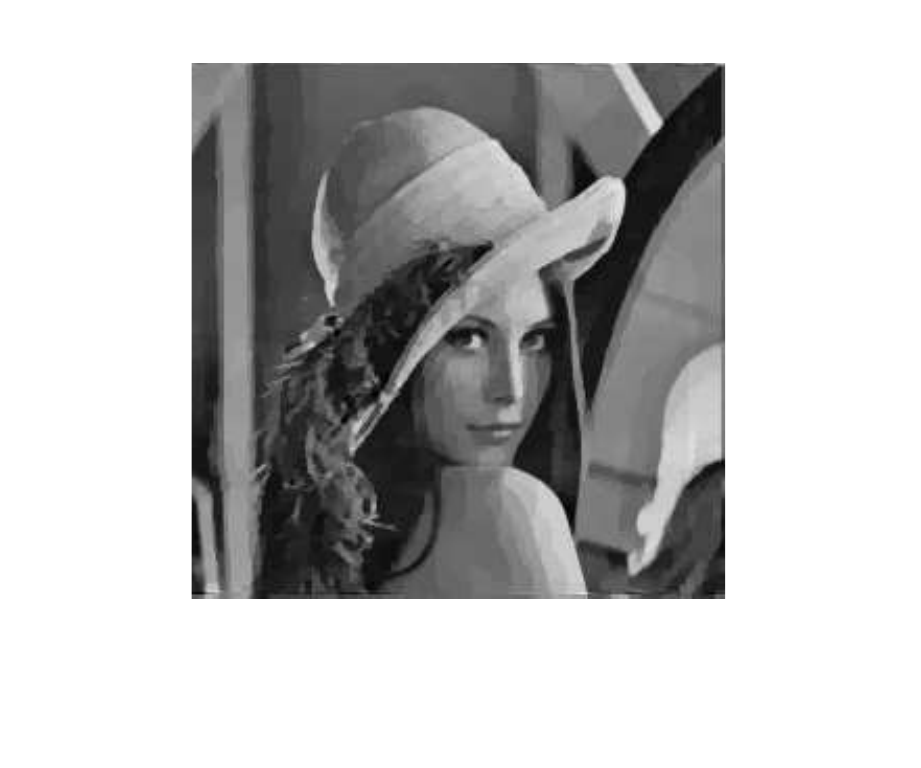}}\\
          \subfloat[]{
    \includegraphics[width=2.4in]{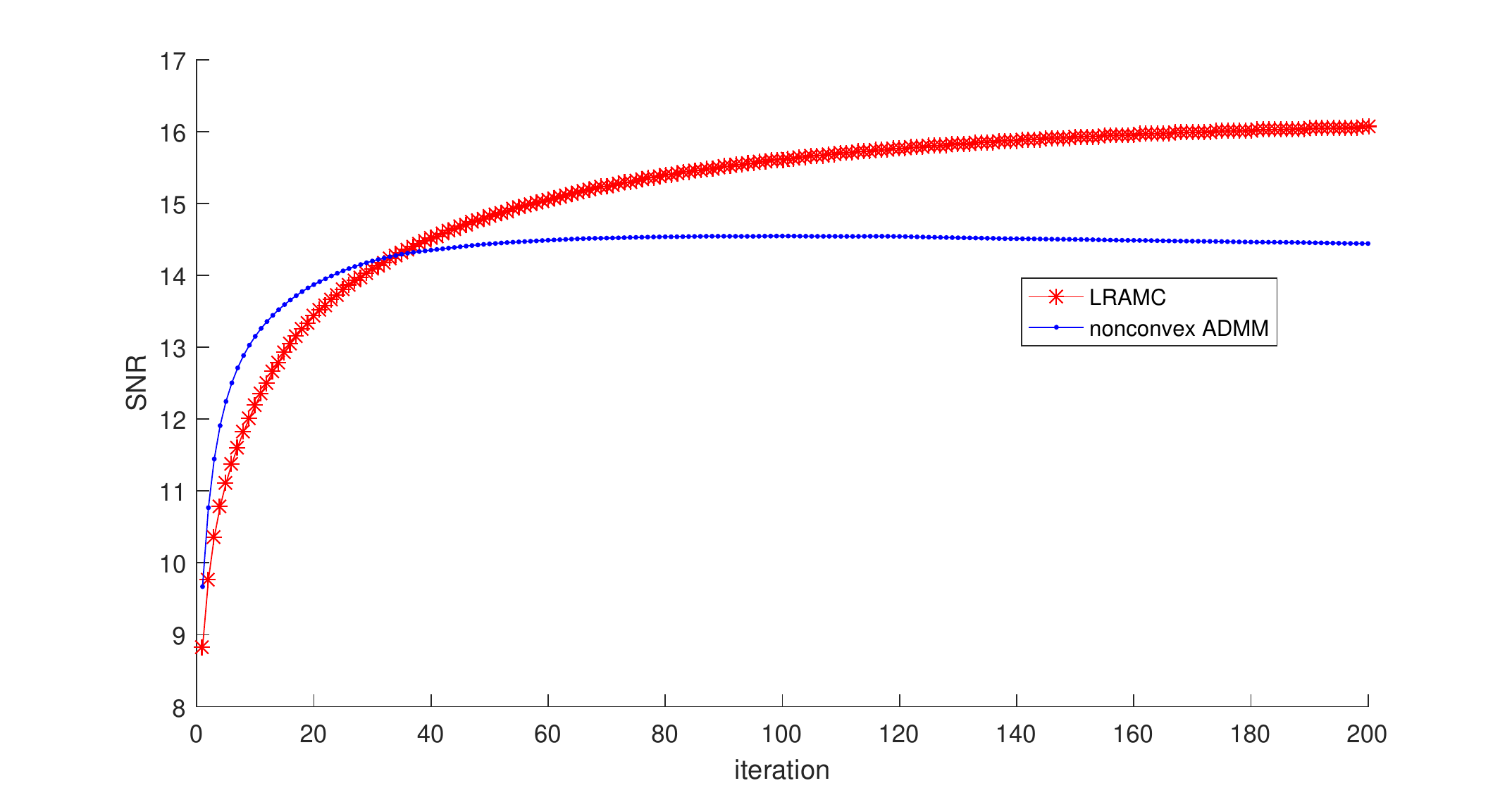}}
\caption{Deblurring results for Lena under Gaussian operator by using the two algorithms. (a) Original image; (b) Blurred image; (c) IRPAMC 16.0dB; (d) nonconvex ADMM 14.4dB; (e) SNR versus the iterations.}
\end{figure}
\section{Conclusion}
In this paper, we propose an iteratively  reweighted alternating minimization algorithm for a class of linearly constrained  problems. The algorithm is developed from the perspective of penalty strategy. To speed up the iteration, we also employ a continuation trick for the penalty parameter. We prove the convergence of the algorithm under weaker assumptions than the nonconvex ADMM. Numerical results on the nonconvex TV deblurring problem are also presented for demonstrating the efficiency of the proposed algorithm.



\end{document}